\documentclass[12pt]{amsart}
\usepackage{amsmath,amssymb,latexsym,cancel,rotating}
\usepackage{graphicx,amssymb,mathrsfs,amsmath,color,fancyhdr,amsthm}
\usepackage[all]{xy}
\usepackage{verbatim} 
\usepackage[all]{xy}
\usepackage{graphicx}
\usepackage{rotating} 
\usepackage{tikz-cd}
\usepackage{setspace}

\newtheorem{theorem}{Theorem}[section]
\newtheorem{lemma}[theorem]{Lemma}

\newtheorem{definition}[theorem]{Definition}
\newtheorem{proposition}[theorem]{Proposition}

\newtheorem{remark}[theorem]{Remark}

\newtheorem{example}[theorem]{Example}

  \def\leq{\leqslant}  \def\geq{\geqslant}

\begin{document}

\title[Iwasawa Decomposition and Total Positivity]
{Notes on Chevalley Groups and Root Category \uppercase\expandafter{\romannumeral4}: Iwasawa Decomposition and Total Positivity}

\author[Buyan Li]{Buyan Li}
\address{Department of Mathematical Sciences, Tsinghua University, Beijing 100084, P. R. China}
\email{liby21@mails.tsinghua.edu.cn}




\bibliographystyle{abbrv}

\maketitle

\begin{abstract}
In this paper, we study the Cartan decomposition and Iwasawa decomposition of the Chevalley group from the root category.
Then we investigate the Iwasawa decomposition for totally positive submonoid of the Chevalley group.
\end{abstract}

\setcounter{tocdepth}{1}\tableofcontents

\section{Introduction}
Chevalley groups are a class of simple groups over arbitrary fields, which generalize the classical Lie groups. 
Based on the Gabriel theorem for representation theory of quivers, Ringel \cite{RINGEL1990137} realizes the positive part of a Lie algebra, and Peng and Xiao \cite{1997ROOT} realize the whole simple Lie algebra using the root category.
Then Xiao and the present author construct the Chevalley group of the root category in \cite{notes1}.
As a continuation of the study of maximal compact subgroups of Chevalley groups in \cite{notes2}, in this paper, we prove the Cartan decomposition and the Iwasawa decomposition theorems for both the Lie algebra and the Chevalley group from the root category.

Lusztig generalizes the theory of totally positive real matrices to any split reductive algebraic group over some field in \cite{TR1,TR2}, establishes nontrivial connection between totally positive monoids and canonical bases, and obtains numerous results.
In \cite{notes3}, we apply Lusztig's theory to Chevalley groups and describe the size of totally positive monoids in terms of the root subgroups corresponding to indecomposable objects in the root category.
In this paper, we study the Iwasawa decomposition of totally positive monoid.
We give a explicit (inductive) algorithm for computing certain components of the totally positive monoid under the Iwasawa decomposition.
As a corollary, we show that the intersection of the totally positive monoid and the maximal compact subgroup is trivial.

\section{Root Categories and Chevalley Groups}
In this section, we recall the construction of simple Lie algebra \cite{1997ROOT} and Chevalley group \cite{notes1} from the root category.

Let $A$ be a finite dimensional associative representation-finite hereditary algebra over some base field $k$, and $\operatorname{mod}A$ be the category of finite dimensional $A$-modules. 
Let $D^b(A)$ be the bounded derived category of $\operatorname{mod}A$ with shift functor denoted by $T$. 
The root category $\mathcal{R}$ of $A$ is the orbit category $D^b(A)/T^2$.
Note that both $D^b(A)$ and $\mathcal{R}$ are triangulated categories, with translation functor $T$.
Let $\operatorname{ind}\mathcal{R}$ be the set of all the indecomposable objects of $\mathcal{R}$. For any object $M\in \mathcal{R}$, we denote its isomorphism class by $[M]$.

Let $\mathcal{K}$ be the Grothendieck group of $\mathcal{R}$, that is, an abelian group defined by generators $H_{[M]}$, for all $M\in \mathcal{R}$, and subject to relations $H_{[X]}+H_{[Z]}=H_{[Y]}$, if there exists a triangle $X\rightarrow Y\rightarrow Z \rightarrow TX$ in $\mathcal{R}$.
Let $\mathcal{N}$ be a free abelian group with a basis $\{u_{[X]}\}_{X\in \operatorname{ind}\mathcal{R}}$. For simplicity, we write $H_M$ instead of $H_{[M]}$, $u_X$ instead of $u_{[X]}$.
Let $(-|-)$ be the symmetric Euler form on $\mathcal{K}$, that is 
\begin{align*}
    (H_X|H_Y)=&\operatorname{dim}_k\operatorname{Hom}_{\mathcal{R}}(X,Y)-\operatorname{dim}_k\operatorname{Hom}_{\mathcal{R}}(X,TY) \\
        &+\operatorname{dim}_k\operatorname{Hom}_{\mathcal{R}}(Y,X)-\operatorname{dim}_k\operatorname{Hom}_{\mathcal{R}}(Y,TX)
\end{align*}
for any $X,Y\in \mathcal{R}$.
For $X\in \operatorname{ind}\mathcal{R}$, let $d(X)=\operatorname{dim}_k\operatorname{End}_{\mathcal{R}}(X)$.

Peng and Xiao \cite{1997ROOT} prove the existence of Hall polynomials in the root category, and use the evaluation of Hall polynomials to define the Lie bracket. 
Then they realize the $\mathbb{Z}$-form $\mathfrak{g}_{\mathbb{Z}}$ of the complex simple Lie algebra $\mathfrak{g}_{\mathbb{C}}=(\mathcal{K}\oplus \mathcal{N})\otimes \mathbb{C}$, whose roots  corresponds bijectively to the indecomposable objects of $\mathcal{R}$.

Based on their work, we construct the Chevalley groups from the root category in \cite{notes1}.
For any $X\in \operatorname{ind}\mathcal{R}$ and $t\in \mathbb{C}$, we define an automorphism $E_{[X]}(t)=\operatorname{exp}(t \operatorname{ad}u_X)$ of the Lie algebra $\mathfrak{g}_{\mathbb{C}}$.
Since the coefficients of the action of $E_{[X]}(t)$ on a Chevalley basis of $\mathfrak{g}_{\mathbb{C}}$ are in $\mathbb{Z}[t]$, we can extend the definition of $E_{[X]}(t)$ to any field $\mathbb{K}$.
Let $\mathfrak{g}_{\mathbb{K}}=\mathfrak{g}_{\mathbb{Z}}\otimes \mathbb{K}$, and $\mathbf{G}=\mathbf{G}(\mathcal{R},\mathbb{K})$ be the subgroup of $\operatorname{Aut}(\mathfrak{g}_{\mathbb{K}})$ generated by $E_{[X]}(t)$, for all $X\in \operatorname{ind}\mathcal{R}$ and $t\in \mathbb{K}$.
This group is called the Chevalley group of the root category $\mathcal{R}$ over field $\mathbb{K}$.

For each $X\in \operatorname{ind}\mathcal{R}$ and $t\in \mathbb{K}^{\times}$, define automorphism $h_{[X]}(t)$ of $\mathfrak{g}_{\mathbb{K}}$ by its action on the basis:
\begin{align*}
    &h_{[X]}(t)u_{Y} = t^{A_{XY}} u_{Y}, \\
    &h_{[X]}(t)H_{S} = H_{S},
\end{align*}
  where $A_{XY}=\frac{2(H_X|H_Y)}{(H_X|H_X)}$ is an integer.

\begin{lemma}\cite[Lemma 4.10]{notes1}\label{hEh}
    For any  $ X,Y \in \operatorname{ind}\mathcal{R},  t \in \mathbb{K}^{\times}, s \in \mathbb{K}$,
    \begin{align*}
        h_{[X]}(t) E_{[Y]}(s) h_{[X]}(t)^{-1} = E_{[Y]}(t^{A_{XY}}s).
    \end{align*}
\end{lemma}

For each $X\in \operatorname{ind}\mathcal{R}$, we use BGP reflection functor to define an automorphism $n_{[X]}$ of $\mathfrak{g}_{\mathbb{K}}$.
We have proved that $n_{[X]},h_{[X]}(t)$ are in $\mathbf{G}$.
Let $\mathbf{H}$ be the subgroup of $\mathbf{G}$ generated by $h_{[M]}(t)$, for all $M \in \operatorname{ind}\mathcal{R}, t \in \mathbb{K}^{\times}$, and let $\mathbf{N}$ be the subgroup of $\mathbf{G}$ generated by $\mathbf{H}$ and $n_{[M]},\forall M \in \operatorname{ind}\mathcal{R}$.
We have $\mathbf{H}$ is a normal subgroup of $\mathbf{N}$, and $\mathbf{N}/\mathbf{H}\cong \mathbf{W}$, where $\mathbf{W}$ is the Weyl group associated to the root system of $\mathcal{R}$. 

If we arbitrarily fix a complete section of the root category $\mathcal{R}$, that is, a connected full subquiver of the AR-quiver of $\mathcal{R}$, which contains one vertex for each $\tau$-orbit ($\tau$ is the AR-translation functor), then we get a complete hereditary subcategory $\mathcal{B}$ of $\mathcal{R}$.
We fix a set of representatives $\{S_1,\cdots,S_n\}$ of isomorphism classes of simple objects in $\mathcal{B}$.

Note that $\mathbf{H}$ is an abelian subgroup of $\mathbf{G}$, and each element in $\mathbf{H}$ is of the form $\prod\limits_{i=1}^n h_{[S_i]}(t_i)$, $t_i\in \mathbb{K}^{\times}$.
Moreover, we have 
 $ \prod\limits_{i=1}^n h_{[S_i]}(t_i) =1$ if and only if 
       $ (t_1,\cdots,t_n) \in (\mathbb{K}^{\times})^n$ satisfies $\prod\limits_{i=1}^n t_i^{a_{ij}} = 1$, for all $ j=1,\cdots,n$, where $(a_{ij})$ is the Cartan matrix, and $a_{ij}=A_{S_i,S_j}$.
In particular, if $(t_1,\cdots,t_n)\in \mathbb{R}_{>0}^n$, then by taking the logarithm, we have $ \prod\limits_{i=1}^n h_{[S_i]}(t_i) =1$ if and only if $t_1=\cdots=t_n=1$.

Let $\mathbf{U}^+$ (resp. $\mathbf{U}^-$) be the subgroup of $\mathbf{G}$ generated by $E_{[X]}(t)$, for all $t\in \mathbb{K}$ and $X\in \operatorname{ind}\mathcal{B}$ (resp. $X\in \operatorname{ind}T\mathcal{B}$).
Let $\mathbf{B}^+$ (resp. $\mathbf{B}^-$) be the subgroup of $\mathbf{G}$ generated by $\mathbf{H}$ and $\mathbf{U}^+$ (resp. $\mathbf{U}^-$).

\begin{proposition}
    \cite[Cor 4.42,4.44]{notes1}
    Assume $\mathbb{K}$ is an algebraically closed field, then $\mathbf{G}$ is a semisimple linear algebraic group, and the Lie algebra of $\mathbf{G}$ is isomorphic to $\mathfrak{g}_{\mathbb{K}}$.
\end{proposition}
In this case, $\mathbf{B}^+,\mathbf{B}^-$ is a pair of opposite Borel subgroups of $\mathbf{G}$, whose intersection is the maximal torus $\mathbf{H}$, and $\mathbf{U}^+,\mathbf{U}^-$ are their unipotent radicals, respectively.

\section{Cartan Decomposition and Iwasawa Decomposition}
In this section, we obtain the Cartan decomposition and the Iwasawa decomposition for the simple Lie algebra and the Chevalley group from the root category $\mathcal{R}$, with the field $\mathbb{K}$ taken to be $\mathbb{C}$.
For the classical results of these decompositions, see reference \cite{Knapp}; however, note that in our case, these decompositions are more intuitional and simpler.

From now on, we denote $\mathfrak{g}=\mathfrak{g}_{\mathbb{C}}$ and $\mathbf{G}=\mathbf{G}(\mathcal{R},\mathbb{C})$.

\subsection{Cartan Decomposition}

In \cite{notes2}, we have essentially obtained the Cartan decomposition for both the Lie algebra and the Chevalley group, though not stated explicitly.
We now present it in full, supplementing the necessary details.

Recall that the shift functor $T$ of the root category $\mathcal{R}$ induces an $\mathbb{R}$-linear involution $\theta$ of $\mathfrak{g}$, given by 
\begin{align*}
    \theta(H_X)=H_{TX}, \quad \theta(u_X)=u_{TX},
\end{align*}
for each $X\in \operatorname{ind}\mathcal{R}$, and $\theta$ maps the imaginary unit $i$ to $-i$. 
We call $\theta$ the Cartan involution of $\mathfrak{g}$.

Let 
\begin{align*}
    \mathfrak{r}=\{ X\in \mathfrak{g}| \theta(X)=X \}
\end{align*}
be the $\theta$-fixed points of $\mathfrak{g}$, which is a real Lie algebra.
By an approach analogous to that in \cite{1997ROOT}, we construct in \cite{notes2} a $\mathbb{Z}$-form of the real Lie algebra $\mathfrak{r}$ from the root category $\mathcal{R}$.
As a result, we obtain a $\mathbb{Z}$-basis of $\mathfrak{r}$. 
Precisely, if we choose a subset $\Delta$ of $\operatorname{ind}\mathcal{R}$ such that $\{H_X,X\in \Delta\}$ form a basis of $\mathcal{K}$, (for example, the set of simple objects in a complete hereditary subcategory), then 
\begin{align*}
    \{i\frac{H_X}{d(X)},X\in \Delta\}\bigcup \{u_X+u_{TX},i(u_X-u_{TX}),X\in \operatorname{ind}\mathcal{R}\}
\end{align*}
form a $\mathbb{Z}$-basis of $\mathfrak{r}$. 
Moreover, we have the following result.

\begin{proposition}
    \cite[Proposition 30]{notes2}
    The real Lie algebra $\mathfrak{r}$ is a compact real form of $\mathfrak{g}$.
\end{proposition}

Since $\theta$ is an involution, it has eigenvalues $\pm 1$.
Let $\mathfrak{p}$ be the eigenspace corresponding to eigenvalue -1. 
Then $\mathfrak{g}$ has an eigenspace decomposition
\begin{align*}
    \mathfrak{g}=\mathfrak{r}\oplus \mathfrak{p}.
\end{align*}
This is called the Cartan decomposition of $\mathfrak{g}$.
Note that $\mathfrak{r}$ is $\mathbb{R}$-spanned by $i\frac{H_X}{d(X)}$, $u_X+u_{TX},i(u_X-u_{TX})$, $X\in \operatorname{ind}\mathcal{R}$, and $\mathfrak{p}$ is $\mathbb{R}$-spanned by $\frac{H_X}{d(X)},u_X-u_{TX},i(u_X+u_{TX})$, $X\in \operatorname{ind}\mathcal{R}$.

Now we consider the Cartan decomposition of the Chevalley group $\mathbf{G}$.
Recall that in \cite{notes2} we let $\mathbf{R}$ be the subgroup of $\operatorname{Aut}(\mathfrak{r})$ generated by $\operatorname{exp}(t\operatorname{ad}i\frac{H_X}{d(X)})$, $\operatorname{exp}(t\operatorname{ad}(u_X+u_{TX})),\operatorname{exp}(t\operatorname{ad}i(u_X-u_{TX}))$, with $X\in \operatorname{ind}\mathcal{R}$ and $t\in \mathbb{R}$, and let $\mathbf{K}$ be the identity component of $\mathbf{R}$.
As a corollary of \cite[Remark 37]{notes2}, we actually have $\mathbf{K}=\mathbf{R}$.

\begin{proposition}
    \cite[Proposition 36]{notes2}
    The group $\mathbf{K}$ is a maximal compact subgroup of $\mathbf{G}$. 
    In fact, it coincides with the inner automorphism group  $\operatorname{Int}(\mathfrak{r})$.
\end{proposition}

We define an involution $\Theta$ on $\mathbf{G}$.
Let $\Theta:\mathbf{G}\rightarrow \mathbf{G}$ be a group automorphism, mapping $E_{[X]}(t)$ to $E_{[TX]}(\overline{t})$, for any $X\in \operatorname{ind}\mathcal{R}$ and $t\in \mathbb{C}$.
Here $\overline{t}$ means taking the conjugate of $t$.
It can be easily shown that $\Theta^2=1$, $\Theta(n_{[X]})=n_{[TX]}$, and $\Theta(h_{[X]}(t))=h_{[TX]}(\overline{t})$, for any $X\in \operatorname{ind}\mathcal{R}$ and $t\in \mathbb{C}$.
Moreover, for any $t\in \mathbb{C}$, since
\begin{align*}
     \Theta(\operatorname{exp}(t\operatorname{ad}u_X))&=\Theta(E_{[X]}(t))  =E_{[TX]}(\overline{t})=\operatorname{exp}(\overline{t}\operatorname{ad}u_{TX})=\operatorname{exp}(\operatorname{ad}\theta(tu_X)),\\
     \Theta(\operatorname{exp}(t\operatorname{ad}\frac{H_X}{d(X)}))&=\Theta(h_{[X]}(e^{-t}))=h_{[TX]}(e^{-\overline{t}})=\operatorname{exp}(\overline{t}\operatorname{ad}\frac{H_{TX}}{d(X)})=\operatorname{exp}(\operatorname{ad}\theta(t\frac{H_X}{d(X)})),
\end{align*}
we have $d\Theta=\theta$.

Denote the set of $\Theta$-fixed points in $\mathbf{G}$ by $\mathbf{G}^{\Theta}$.
Then the Lie algebra of $\mathbf{G}^{\Theta}$ is $\mathfrak{r}$.

\begin{lemma}
    \cite[Lemma 31]{notes2}
    For any $t\in \mathbb{R}, X\in \operatorname{ind}\mathcal{R}$, we have 
    \begin{align*}
       & \operatorname{exp}(t\operatorname{ad}(u_X+u_{TX}))=E_{[X]}(\operatorname{tan}t)h_{[X]}((\operatorname{cos}t)^{-1})E_{[TX]}(\operatorname{tan}t),\\
        & \operatorname{exp}(t\operatorname{ad}i(u_X-u_{TX}))=E_{[X]}(i\operatorname{tan}t)h_{[X]}((\operatorname{cos}t)^{-1})E_{[TX]}(-i\operatorname{tan}t).
    \end{align*}
\end{lemma}

Using the definition and the above lemma, we notice that $\mathbf{K}$ is contained in $\mathbf{G}^{\Theta}$.
We denote the identity component by $(-)^{\circ}$, and denote the Lie algebra of the Lie group by $\operatorname{Lie}(-)$.

\begin{proposition}
    The group $\mathbf{K}$ coincides with $\mathbf{G}^{\Theta}$.
\end{proposition}

\begin{proof}
    The proposition follows easily from the facts that $\mathbf{K}=\operatorname{Int}(\mathfrak{r})=\operatorname{Aut}(\mathfrak{r})^{\circ}$, $\operatorname{Lie}(\mathbf{G}^{\Theta})=\mathfrak{r}=\operatorname{Lie}(\mathbf{K})$, and $\mathbf{G}^{\Theta}$ is a connected, closed subgroup of $\operatorname{Aut}(\mathfrak{r})$.
\end{proof}

Recall the definition of an invariant symmetric bilinear form on $\mathfrak{g}$ arising from $\mathcal{R}$.

\begin{definition}\cite[Definition 29]{notes2}\label{(,)}
Let $(,):\mathfrak{g} \times \mathfrak{g} \rightarrow \mathbb{C}$ be a non-degenerate symmetric bilinear form on $\mathfrak{g}$, satisfying the following properties:

(1) For any $x,y,z\in \mathfrak{g}$, we have 
\begin{align*}
    ([x,y],z)=(x,[y,z]).
\end{align*}
    
(2) For any $X,Y\in \mathcal{R}$, we have 
\begin{align*}
    (H_X,H_Y)=\sum_{Z\in \operatorname{ind}\mathcal{R}}(H_X|H_Z)(H_Y|H_Z),
\end{align*}
where $(-|-)$ is the symmetric Euler form on $\mathcal{K}$.

(3) For any $X,Y\in \operatorname{ind}\mathcal{R}$, $(H_X,u_Y)=0$, i.e. $(\mathcal{K},\mathcal{N})=0$.

(4) For any $X,Y\in \operatorname{ind}\mathcal{R}$, if $X\ncong TY$, then $(u_X,u_Y)=0$. 
If $X\cong TY$, then 
\begin{align*}
    (u_X,u_{TX})= -4+\sum_{Y\in\operatorname{ind}\mathcal{R}}(\sum_{L\in\operatorname{ind}\mathcal{R}} \gamma_{TX,Y}^L\gamma_{X,L}^Y),
\end{align*}
where the Hall numbers $\gamma_{MN}^P=\varphi_{MN}^P(1)-\varphi_{NM}^P(1) \in \mathbb{Z}$ are evaluations of the Hall polynomials $\varphi_{MN}^P,\varphi_{NM}^P$ and are used to define the Lie bracket by Ringel, Peng and Xiao. 
\end{definition}

\begin{lemma}
    For any $x\in \mathbf{G}$ and $Y,Z\in \mathfrak{g}$, we have 
    \begin{align*}
        (xY,xZ)=(Y,Z).
    \end{align*}
\end{lemma}

\begin{proof}
    It suffices to prove for generators $E_{[X]}(t)$ of the Chevalley group.
    It follows from Definition \ref{(,)}(1) that $((\operatorname{ad}y)(x),z)=-(x,(\operatorname{ad}y)(z))$, for any $x,y,z\in \mathfrak{g}$.
    Now for $X\in \operatorname{ind}\mathcal{R},t\in \mathbb{C},Y,Z\in \mathfrak{g}$, we have 
    \begin{align*}
        (E_{[X]}(t)Y,Z)&=((1+t\operatorname{ad}u_X+\frac{t^2}{2!}(\operatorname{ad}u_X)^2+\cdots)Y,Z)\\
       & =(Y,(1-t\operatorname{ad}u_X+\frac{(-t)^2}{2!}(\operatorname{ad}u_X)^2+\cdots)Z)=(Y,E_{[X]}(-t)Z),
    \end{align*}
    and the lemma is proved.
\end{proof}

By definition, it is easy to check that $(\theta (X),\theta (Y))=(X,Y)$ for any $X,Y\in \mathfrak{g}$.

Now we define a bilinear form $(-,-)_{\theta}$ on $\mathfrak{g}$ by 
\begin{align*}
    (X,Y)_{\theta}=-(X,\theta (Y)),
\end{align*}
for any $X,Y\in \mathfrak{g}$.
Then 
\begin{align*}
    (Y,X)_{\theta}=-(Y,\theta(X))=-(\theta(Y),X)=-(X,\theta (Y))=(X,Y)_{\theta}
\end{align*}
and $(-,-)_{\theta}$ is symmetric.

\begin{lemma}
    The symmetric bilinear form $(-,-)_{\theta}$ is positive definite.
\end{lemma}

\begin{proof}
    By definition, we notice that $(-,-)_{\theta}$ coincides with $(-,-)$ when restrict to $\mathcal{K}$, and $(\mathcal{K},\mathcal{N})_{\theta}=0$.
    For any $X,Y\in \operatorname{ind}\mathcal{R}$, if $X\ncong Y$, then 
    \begin{align*}
        (u_X,u_Y)_{\theta}=-(u_X,u_{TY})=0.
    \end{align*}
    Otherwise,
    \begin{align*}
        (u_X,u_X)_{\theta}=-(u_X,u_{TX})>0,
    \end{align*}
    which is shown in the proof of \cite[Proposition 30]{notes2} by properties of Hall numbers $\gamma$ and triangles in $\mathcal{R}$.
    Thus $(-,-)_{\theta}$ is positive definite.
\end{proof}

We regard $\mathfrak{g}$ as a real Lie algebra, and then $(-,-)_{\theta}$ is an inner product on it.

\begin{lemma}
    The Chevalley group $\mathbf{G}$ coincides with $\operatorname{Int}(\mathfrak{g})$.
\end{lemma}

\begin{proof}
    Since $\operatorname{Int}(\mathfrak{g})$ is generated by $\operatorname{exp}(\operatorname{ad}X), X\in \mathfrak{g}$, $\mathbf{G}$ is contained in $\operatorname{Int}(\mathfrak{g})$.
    On the other hand, the Lie algebras $\operatorname{Lie}(\mathbf{G})=\mathfrak{g}=\operatorname{Lie}(\operatorname{Int}(\mathfrak{g}))$, so $\mathbf{G}^{\circ}=\operatorname{Int}(\mathfrak{g})^{\circ}=\operatorname{Int}(\mathfrak{g})$, and the lemma follows.
\end{proof}

\begin{theorem}
    The map $\mathbf{K}\times \mathfrak{p}\rightarrow \mathbf{G}, (k,X)\mapsto k\operatorname{exp}(\operatorname{ad}X)$ is a diffeomorphism.
\end{theorem}

\begin{proof}
    We first show that this map is surjective.
    For any $x\in \mathbf{G}$, we consider the element $y=\Theta(x)^{-1}x$.
    Then $\Theta(y)=y^{-1}$.
    Moreover, for any $X,Y,Z\in \mathfrak{g}$, we have 
    \begin{align*}
        (yX,Y)_{\theta}&=-(yX,\theta(Y))=-(X,y^{-1}\theta(Y))\\
        &=-(X,\Theta(y)\theta(Y))=-(X,\theta(yY))=(X,yY)_{\theta},
    \end{align*}
and
\begin{align*}
    (yZ,Z)_{\theta}=-(\Theta(x)^{-1}xZ,\theta(Z))=-(xZ,\theta(xZ))=(xZ,xZ)_{\theta}\geq 0.
\end{align*}
Here $ (yZ,Z)_{\theta}=0$ if and only if $xZ=0$, if and only if $Z=0$.
As a result, the element $y$, viewed as an automorphism of $\mathfrak{g}$, is diagonalizable and all its eigenvalues are positive real numbers.
We choose a basis of $\mathfrak{g}$ made of eigenvectors of $y$ and denote the eigenvalues of $y$ by $\lambda_1,\cdots,\lambda_m$.
With respect to this basis, $y$ becomes a diagonal matrix $\operatorname{diag}\{\lambda_1,\cdots,\lambda_m\}$.
For any $t\in \mathbb{R}$, define $y^t=\operatorname{diag}\{\lambda_1^t,\cdots,\lambda_m^t\}$, which is an automorphism of $\mathfrak{g}$.
The elements $\{y^t,t\in \mathbb{R}\}$ form a one-parameter subgroup of $\operatorname{Aut}(\mathfrak{g})$, and thus contained in $\operatorname{Aut}(\mathfrak{g})^{\circ}=\operatorname{Int}(\mathfrak{g})=\mathbf{G}$.
So $\{y^t,t\in \mathbb{R}\}$ is a one-parameter subgroup of $\mathbf{G}$, and there exists some $X\in \mathfrak{g}$ such that $y^t=\operatorname{exp}(t\operatorname{ad}X)$ for any $t\in \mathbb{R}$.

Consider the one-parameter subgroups $\{\Theta(\operatorname{exp}(t\operatorname{ad}X)),t\in \mathbb{R}\}$ and $\{\operatorname{exp}(-t\operatorname{ad}X)$, $t\in \mathbb{R}\}$. 
Since $\Theta(y)=y^{-1}$, we have $\Theta(\operatorname{exp}(\operatorname{ad}X))=\operatorname{exp}(-\operatorname{ad}X)$.
That is, these two one-parameter subgroups coincides at $t=1$, and thus are equal.
By taking differential at $t=0$, we obtain that $\theta(X)=-X$ and $X\in \mathfrak{p}$.

Let $k=x\operatorname{exp}(-\frac{1}{2}\operatorname{ad}X)$.
Then 
\begin{align*}
    \Theta(k)^{-1}k&=\operatorname{exp}(-\frac{1}{2}\operatorname{ad}X) \Theta(x)^{-1}x \operatorname{exp}(-\frac{1}{2}\operatorname{ad}X)\\
    &=\operatorname{exp}(-\frac{1}{2}\operatorname{ad}X)\operatorname{exp}(\operatorname{ad}X)\operatorname{exp}(-\frac{1}{2}\operatorname{ad}X)=1.
\end{align*}
Thus $\Theta(k)=k$ and $k\in \mathbf{K}$.
So $x=k\operatorname{exp}(\frac{1}{2}\operatorname{ad}X)$ is the image of $(k,\frac{1}{2}X)\in \mathbf{K}\times \mathfrak{p}$ and the map is surjective.

Now we show that the map is injective.
By the previous argument, any $x\in \mathbf{G}$ can be written of the form $k\operatorname{exp}(\operatorname{ad}X)$ for some $k\in \mathbf{K}$ and $X\in \mathfrak{p}$. 
Repeating the same process as above, the eigenvectors of $\Theta(x)^{-1}x$ corresponding to the eigenvalues $\lambda_1,\cdots,\lambda_m$ form a basis of $\mathfrak{g}$.
Then with respect to this basis,
\begin{align*}
   X=\operatorname{diag}\{\operatorname{ln}\lambda_1,\cdots,\operatorname{ln}\lambda_m\}
\end{align*}
 is uniquely determined.
Thus $k$ is also uniquely determined and the map is injective.

Obviously the map is smooth. For the smoothness of its inverse map, we refer the proof of \cite[Theorem 6.31]{Knapp} for details.
The proof is finished.
\end{proof}

The diffeomorphism in the above theorem is called the Cartan decomposition of $\mathbf{G}$, and $\Theta$ is called the Cartan involution of $\mathbf{G}$.

\subsection{Iwasawa Decomposition}
In this subsection, we study the Iwasawa decomposition for the Lie algebra $\mathfrak{g}$ and the Chevalley group $\mathbf{G}$.

Let $\mathfrak{a}$ be the $\mathbb{R}$-span of $H_X,X\in \mathcal{R}$. 
Then $\mathfrak{a}$ is a maximal abelian subspace of $\mathfrak{p}$.

\begin{remark}
    For $\lambda\in\mathfrak{a}^*$, let 
    \begin{align*}
        \mathfrak{g}_{\lambda}=\{X\in \mathfrak{g}|[H,X]=\lambda(H)X, \forall H\in \mathfrak{a}  \}.
    \end{align*}
    If $\lambda\neq 0$ and $\mathfrak{g}_{\lambda}\neq \{0\}$, then $\lambda$ is called a restricted root of $\mathfrak{g}$.
    The set of restricted roots $\Sigma$ form an abstract root system in $\mathfrak{a}^*$.
    In general, $\Sigma$ is not equal to $\Phi$.
    However, in our case, the restricted roots corresponds bijectively to the roots of $\mathfrak{g}$, and the two root systems are equal.
\end{remark}

Now we arbitrarily fix a complete section of $\mathcal{R}$, and get a hereditary subcategory $\mathcal{B}$ of $\mathcal{R}$.
We fix a complete set of representatives $\{S_1,\cdots,S_n\}$ of isomorphism classes of simple objects in $\mathcal{B}$.

Let $\mathfrak{u}_+$ be the Lie subalgebra of $\mathfrak{g}$, $\mathbb{C}$-spanned by $u_X,X\in \operatorname{ind}\mathcal{B}$.
Then $\mathfrak{u}_+$ is the Lie algebra of $\mathbf{U}^+$.

\begin{proposition}
    We have an $\mathbb{R}$-vector space decomposition
    \begin{align*}
        \mathfrak{g}=\mathfrak{r}\oplus \mathfrak{a} \oplus \mathfrak{u}_+.
    \end{align*}
    This is called the Iwasawa decomposition for $\mathfrak{g}$.
\end{proposition}

\begin{proof}
    Notice that $\mathfrak{r}$ has an $\mathbb{R}$-basis
    \begin{align*}
         \{iH_{S_1},\cdots,iH_{S_n};u_X+u_{TX},i(u_X-u_{TX}),X\in \operatorname{ind}\mathcal{R}\},
    \end{align*}
    $\mathfrak{a}$ has an $\mathbb{R}$-basis $\{H_{S_1},\cdots,H_{S_n}\}$, and $\mathfrak{u}_+$ has an $\mathbb{R}$-basis $\{u_X,iu_X,X\in \operatorname{ind}\mathcal{B}\}$.
     The proposition follows.
\end{proof}

Let $\mathbf{A}$ be the subgroup of $\mathbf{G}$ generated by $h_{[X]}(t),X\in\operatorname{ind}\mathcal{R},t\in \mathbb{R}_{>0}$.
Then it is simply-connected and has Lie algebra $\mathfrak{a}$.

\begin{theorem}
    The multiplication $\mathbf{K}\times \mathbf{A}\times \mathbf{U}^+ \rightarrow \mathbf{G}, (k,a,u)\mapsto kau$ is a diffeomorphism.
    This is called the Iwasawa decomposition for $\mathbf{G}$.
\end{theorem}

\begin{proof}
    By definition, $\mathbf{A}$ is a closed subgroup of $\mathbf{H}$.
    Using the properties of $\mathbf{B}^+=\mathbf{U}^+\mathbf{H}$, the multiplication $\mathbf{A}\times \mathbf{U}^+\rightarrow \mathbf{A}\mathbf{U}^+$ is bijective.
    The subgroup $\mathbf{A}\mathbf{U}^+$ is closed, and has Lie algebra $\mathfrak{a}\oplus \mathfrak{u}_+$.
    Considering differentials, the multiplication $\mathbf{A}\times \mathbf{U}^+\rightarrow \mathbf{A}\mathbf{U}^+$ is a diffeomorphism.

    We can regard the multiplication map $\mathbf{K}\times \mathbf{A}\times \mathbf{U}^+ \rightarrow \mathbf{G}$ as the composition of maps $\mathbf{K}\times \mathbf{A}\times \mathbf{U}^+ \rightarrow \mathbf{K}\times \mathbf{A} \mathbf{U}^+ \rightarrow \mathbf{G} $.
    The image is the product of a compact subset and a closed subset, and is thus closed.
    On the other hand, by the Iwasawa decomposition of $\mathfrak{g}$, the differential $\mathfrak{r}\times (\mathfrak{a}\oplus \mathfrak{u}_+)\rightarrow \mathfrak{g},(X,Y)\mapsto X+Y$ is an isomorphism.
    Then the image of the multiplication map is open by the submersion theorem.
    Since $\mathbf{G}$ is connected, the image $\mathbf{K} \mathbf{A} \mathbf{U}^+=\mathbf{G}$ and the multiplication map is surjective.

    By definition, we have $\Theta(\mathbf{A} \mathbf{U}^+)=\mathbf{A} \mathbf{U}^-$, and $\Theta(h)=h^{-1}$, for any $h\in \mathbf{A}$.
    Since $\mathbf{A} \mathbf{U}^+ \bigcap \mathbf{A} \mathbf{U}^-=\mathbf{A}$, the only $\Theta$-fixed point in $\mathbf{A} \mathbf{U}^+$ is 1.
    We have shown that $\mathbf{K}=\mathbf{G}^{\Theta}$, and thus $\mathbf{K}\bigcap \mathbf{A} \mathbf{U}^+ =\{1\}$.
    This implies that the multiplication map is injective.
    Then it is easy to see that the multiplication map is a diffeomorphism and the theorem is proved.
\end{proof}

\section{Iwasawa Decomposition for Totally Positive Monoid}
In this section, we first recall Lusztig's theory of total positivity and its application to Chevalley groups.
Then we study the Iwasawa decomposition for totally positive elements. 
Precisely, we calculate the $\mathbf{A}$ and $\mathbf{U}^+$ component for each totally positive element in Iwasawa decomposition, and show that the intersection of the totally positive monoid and $\mathbf{K}$ is trivial.

\subsection{Totally Positive Monoid in Chevalley group}
For any split reductive connected algebraic group over some infinite field, Lusztig defines the totally positive submonoids, and obtain numerous results concerning canonical basis, flag varieties, and related topics, see for instance \cite{TR1,TC,TR2}.
Lusztig's theory can be naturally applied to Chevalley groups.
In this subsection, we recall the definition and some results, and refer \cite{notes3} for details.

\begin{lemma}\cite[Lemma 3.1]{notes3} \label{xhy=yhx}
    For any $X\in \operatorname{ind}\mathcal{R}$, $a,c\geq 0$ and $b>0$, we have
    \begin{align*}
        E_{[X]}(a) h_{[X]}(b) E_{[TX]}(-c) = E_{[TX]}(-\frac{c}{ac+b^2}) h_{[X]}(\frac{ac+b^2}{b}) E_{[X]}(\frac{a}{ac+b^2}).
    \end{align*}
\end{lemma}

Again, we arbitrarily fix a complete section of $\mathcal{R}$, and get a complete hereditary subcategory $\mathcal{B}$ of $\mathcal{R}$.
Let $\{S_1,\cdots,S_n\}$ be a complete set of representatives of isomorphism classes of simple objects in $\mathcal{B}$.

Let $\mathbf{U}^+_{\geq 0}$ be the submonoid of $\mathbf{U}^+$ generated by $E_{[S_i]}(t)$ for $t\geq 0$, $i=1,\cdots,n$.
Let $\mathbf{U}^-_{\geq 0}$ be the submonoid of $\mathbf{U}^-$ generated by $E_{[TS_i]}(-t)$ for $t\geq 0$, $i=1,\cdots,n$.
Let $\mathbf{H}_{>0}$ be the submonoid of $\mathbf{H}$ generated by $h_{[S_i]}(t)$ for $t>0$, $i=1,\cdots,n$, which coincides with $\mathbf{A}$.
Let $\mathbf{G}_{\geq 0}$ be the submonoid of $\mathbf{G}$ generated by $E_{[S_i]}(t), E_{[TS_i]}(-t)$ for $t\geq 0$, and $h_{[S_i]}(t)$ for $t>0$, $i=1,\cdots,n$.
Then $\mathbf{G}_{\geq 0}=\mathbf{U}^+_{\geq 0}\mathbf{H}_{>0}\mathbf{U}^-_{\geq 0}=\mathbf{U}^-_{\geq 0}\mathbf{H}_{>0}\mathbf{U}^+_{\geq 0}$.
By \cite[Proposition 3.3]{notes3}, the submonoid $\mathbf{G}_{\geq 0}$ is isomorphic to Lusztig's totally positive submonoid.

\begin{lemma}\label{utu}
    \cite[Lemma 2.3]{TR1}

    (1) Any element $x\in \mathbf{G}_{\geq 0}$ can be written uniquely in the form $x=u^+ t u^-$ with $u^+\in \mathbf{U}^+_{\geq 0}$, $t\in \mathbf{H}_{>0}$ and $u^-\in \mathbf{U}^-_{\geq 0}$.

    (2) Any element $x\in \mathbf{G}_{\geq 0}$ can be written uniquely in the form $x=u^- t u^+$ with $u^+\in \mathbf{U}^+_{\geq 0}$, $t\in \mathbf{H}_{>0}$ and $u^-\in \mathbf{U}^-_{\geq 0}$.
\end{lemma}

For any $n\in \mathbf{N}$, we denote its image under the natural projection $\mathbf{N}\rightarrow \mathbf{N}/\mathbf{H}$ by $\overline{n}$.
The length of $\overline{n}$ is 
\begin{align*}
    l(\overline{n})=\operatorname{min} \{ t\geq 0| \overline{n}=\overline{n_{[S_{i_1}]}}\cdots  \overline{n_{[S_{i_t}]}}, S_{i_1},\cdots,S_{i_t} \text{ are simple objects in }\mathcal{B} \},
\end{align*}
and a reduced expression of $\overline{n}$ is a sequence $(i_1,\cdots,i_t)$ such that $ \overline{n}=\overline{n_{[S_{i_1}]}}\cdots  \overline{n_{[S_{i_t}]}}$ and $t=l(\overline{n})$.
Let $\overline{n}_0$ be the longest element in $\mathbf{N}/\mathbf{H}$.

For any $n\in \mathbf{N}$, let $(i_1,\cdots,i_t)$ be a reduced expression for $\overline{n}$.
Define 
\begin{align*}
    &\mathbf{U}^+(\overline{n})=\{ E_{[S_{i_1}]}(a_1)\cdots E_{[S_{i_t}]}(a_t)\mid a_1,\cdots,a_t \in \mathbb{R}_{>0}  \},\\
     &\mathbf{U}^-(\overline{n})=\{ E_{[TS_{i_1}]}(-a_1)\cdots E_{[TS_{i_t}]}(-a_t)\mid a_1,\cdots,a_t \in \mathbb{R}_{>0}  \},
\end{align*}
which are well-defined by the following proposition.

\begin{proposition}\label{un}
    \cite[Proposition 2.7, Corollary 2.8]{TR1}

    (1) For any $n\in \mathbf{N}$, let $(i_1,\cdots,i_t)$ be a reduced expression for $\overline{n}$. 
    The map $\mathbb{R}_{>0}^t \rightarrow \mathbf{U}^+$ given by $(a_1,\cdots,a_t)\mapsto E_{[S_{i_1}]}(a_1)\cdots E_{[S_{i_t}]}(a_t)$ is injective, and its image does not depend on the choice of the reduced expression of $\overline{n}$.

    (2) The image $\mathbf{U}^+(\overline{n})$ is contained in $\mathbf{B}^- n \mathbf{B}^-$.

    (3) We have $\mathbf{U}_{\geq 0}^+=\bigsqcup_{\overline{n}\in \mathbf{N}/\mathbf{H}} \mathbf{U}^+(\overline{n}) $.
\end{proposition}

Similar results hold for $\mathbf{U}^-$.

\begin{lemma}\label{u>=0}
    \cite[Lemma 2.10]{TR1}
    Let $(i_1,\cdots,i_m)$ be any reduced expression of $\overline{n}_0$, then the image of the map $\mathbb{R}_{\geq 0}^m\rightarrow \mathbf{U}^+, (a_1,\cdots,a_m)\mapsto E_{[S_{i_1}]}(a_1)\cdots E_{[S_{i_m}]}(a_m)$ is equal to $\mathbf{U}_{\geq 0}^+$.
\end{lemma}

Entirely analogous result holds for $\mathbf{U}^-$. 
That is, 
\begin{align*}
    \mathbf{U}^-_{\geq 0}=\{ E_{[TS_{i_1}]}(-a_1)\cdots E_{[TS_{i_m}]}(-a_m)|(a_1,\cdots,a_m)\in \mathbb{R}_{\geq 0}^m         \}
\end{align*}
for any reduced expression $(i_1,\cdots,i_m)$ of $\overline{n}_0$.

\subsection{Iwasawa Decomposition for Totally Positive Elements}
In this subsection, we will study the Iwasawa decomposition for $\mathbf{G}_{\geq 0}$.

We keep the notations from the previous sections, and arbitrarily fix a reduced expression $(i_1,\cdots,i_m)$ of $\overline{n}_0$.
Let $(a_{ij})$ be the Cartan martix.
For simplicity, we denote $h_{[S_i]}(t)$ by $h_i(t)$, for $i=1,\cdots,n$ and $t\in \mathbb{C}$.

For any $b_1,\cdots,b_n\in \mathbb{R}_{>0}$ and $c_1,\cdots,c_m\in \mathbb{R}_{\geq 0}$, we will inductively calculate $\alpha_1,\cdots,\alpha_n\in \mathbb{R}_{>0}$ and $\beta_1,\cdots,\beta_m \in \mathbb{R}_{\geq 0}$, such that 
\begin{align*}
    &h_1(b_1)\cdots h_n(b_n) E_{[S_{i_m}]}(c_m) \cdots E_{[S_{i_1}]}(c_1)E_{[TS_{i_1}]}(-c_1)\cdots E_{[TS_{i_m}]}(-c_m) h_1(b_1)\cdots h_n(b_n) \\
    &= E_{[TS_{i_1}]}(-\beta_1)\cdots E_{[TS_{i_m}]}(-\beta_m)h_1(\alpha_1)\cdots h_n(\alpha_n) E_{[S_{i_m}]}(\beta_m)\cdots E_{[S_{i_1}]}(\beta_1).
\end{align*}
Note that the parameters in $E_{[TS_{i_j}]}(-)$ and $E_{[S_{i_j}]}(-)$ in the right hand side of the equation are $\pm \beta_j$. 
If we denote the left hand side of the equation by $x$, then this coincidence follows from $\Theta(x)=x^{-1}$ and $x\in \mathbf{G}_{\geq 0}$.

We firstly consider $E_{[S_{i_m}]}(c_m) \cdots E_{[S_{i_1}]}(c_1)E_{[TS_{i_1}]}(-c_1)\cdots E_{[TS_{i_m}]}(-c_m)$.
By Lemma \ref{xhy=yhx}, we have 
\begin{align*}
    E_{[S_{i_1}]}(c_1)E_{[TS_{i_1}]}(-c_1)=E_{[TS_{i_1}]}(-\frac{c_1}{c_1^2+1})h_{i_1}(c_1^2+1)E_{[S_{i_1}]}(\frac{c_1}{c_1^2+1}),
\end{align*}
and let 
\begin{align*}
    \beta_1^{(1)}=\frac{c_1}{c_1^2+1},\quad \alpha_{i_1}^{(1)}=c_1^2+1.
\end{align*}

For $t>1$, assume that we have calculated $\beta_1^{(t)},\cdots,\beta_t^{(t)}\geq 0$ and $\alpha_{i_1}^{(t)},\cdots,\alpha_{i_t}^{(t)}>0$, such that 
\begin{align*}
    &E_{[S_{i_t}]}(c_t) \cdots E_{[S_{i_1}]}(c_1)E_{[TS_{i_1}]}(-c_1)\cdots E_{[TS_{i_t}]}(-c_t)\\
    =&E_{[TS_{i_1}]}(-\beta_1^{(t)})\cdots E_{[TS_{i_t}]}(-\beta_t^{(t)})h_{i_1}(\alpha_{i_1}^{(t)})\cdots h_{i_t}(\alpha_{i_t}^{(t)}) E_{[S_{i_t}]}(\beta_t^{(t)})\cdots E_{[S_{i_1}]}(\beta_1^{(t)}).
\end{align*}
Then we consider $E_{[S_{i_{t+1}}]}(c_{t+1}) \cdots E_{[S_{i_1}]}(c_1)E_{[TS_{i_1}]}(-c_1)\cdots E_{[TS_{i_{t+1}}]}(-c_{t+1})$.

If $i\neq j$, then $S_i$ and $TS_j$ have no indecomposable extensions, and thus for any $s,r\in \mathbb{C}$, $E_{[S_i]}(s)E_{[TS_j]}(r)=E_{[TS_j]}(r)E_{[S_i]}(s)$.
Repeatedly using this fact and Lemma \ref{xhy=yhx}, (or using an inductive method similar to the proof of \cite[Proposition 3.2]{TC}), for $1\leq k\leq t$, we have:

 if $i_k\neq i_{t+1}$, then 
\begin{align}\label{inductive1}
    \beta_k^{(t+1)}=\beta_k^{(t)}(c_{t+1}\sum_{\substack{i_s=i_{t+1}\\1\leq s<k}} \beta_s^{(t)}+1)^{-a_{i_{t+1},i_k}};
\end{align}

 if $i_k=i_{t+1}$, then 
\begin{align}\label{inductive2}
    \beta_k^{(t+1)}=\beta_k^{(t)}(c_{t+1}\sum_{\substack{i_s=i_{t+1}\\1\leq s<k}} \beta_s^{(t)}+1)^{-1} (c_{t+1}\sum_{\substack{i_s=i_{t+1}\\1\leq s\leq k}} \beta_s^{(t)}+1)^{-1}.
\end{align}
Note that these formulas are similar to the formulas in \cite[Definition 3.6]{notes3}.
Now we have 
\begin{align*}
    &E_{[S_{i_{t+1}}]}(c_{t+1})E_{[TS_{i_1}]}(-\beta_1^{(t)})\cdots E_{[TS_{i_t}]}(-\beta_t^{(t)})\\
    =& E_{[TS_{i_1}]}(-\beta_1^{(t+1)})\cdots E_{[TS_{i_t}]}(-\beta_t^{(t+1)}) \\
    &\times  E_{[S_{i_{t+1}}]}(c_{t+1}(c_{t+1}\sum_{\substack{i_s=i_{t+1}\\1\leq s\leq t}} \beta_s^{(t)}+1)) h_{i_{t+1}}(c_{t+1}\sum_{\substack{i_s=i_{t+1}\\1\leq s\leq t}} \beta_s^{(t)}+1).
\end{align*}
Similar result holds for $ E_{[S_{i_t}]}(\beta_t^{(t)})\cdots E_{[S_{i_1}]}(\beta_1^{(t)}) E_{[TS_{i_{t+1}}]}(-c_{t+1})$.
Putting them together, we have 

\begin{align*}
    &E_{[S_{i_{t+1}}]}(c_{t+1}) \cdots E_{[S_{i_1}]}(c_1)E_{[TS_{i_1}]}(-c_1)\cdots E_{[TS_{i_{t+1}}]}(-c_{t+1})\\
   =&  E_{[TS_{i_1}]}(-\beta_1^{(t+1)})\cdots E_{[TS_{i_t}]}(-\beta_t^{(t+1)})E_{[S_{i_{t+1}}]}(c_{t+1}(c_{t+1}\sum_{\substack{i_s=i_{t+1}\\1\leq s\leq t}} \beta_s^{(t)}+1))  \\
    &\times  h_{i_{t+1}}(c_{t+1}\sum_{\substack{i_s=i_{t+1}\\1\leq s\leq t}} \beta_s^{(t)}+1) h_{i_1}(\alpha_{i_1}^{(t)})\cdots h_{i_t}(\alpha_{i_t}^{(t)}) h_{i_{t+1}}(c_{t+1}\sum_{\substack{i_s=i_{t+1}\\1\leq s\leq t}} \beta_s^{(t)}+1) \\
    &\times E_{[TS_{i_{t+1}}]}(-c_{t+1}(c_{t+1}\sum_{\substack{i_s=i_{t+1}\\1\leq s\leq t}} \beta_s^{(t)}+1))  E_{[S_{i_t}]}(\beta_t^{(t+1)})\cdots E_{[S_{i_1}]}(\beta_1^{(t+1)}).
\end{align*}
By Lemma \ref{hEh} and Lemma \ref{xhy=yhx}, we move $E_{[S_{i_{t+1}}]}(-)$ backward in the formula and move $E_{[TS_{i_{t+1}}]}(-)$ forward.
Notice that for $1\leq k\leq t$, we have $\alpha_{i_k}^{(t+1)}=\alpha_{i_k}^{(t)}$, and we denote it by $\tilde{\alpha}_{i_k}$. 
Then 
\begin{align}\label{beta}
    \beta_{t+1}^{(t+1)}=\frac{c_{t+1} (c_{t+1}\sum_{\substack{i_s=i_{t+1}\\1\leq s\leq t}} \beta_s^{(t)}+1)^{-3} \prod_{k=1}^{t} \tilde{\alpha}_{i_k}^{-a_{i_k,i_{t+1}}} }{c_{t+1}^2 (c_{t+1}\sum_{\substack{i_s=i_{t+1}\\1\leq s\leq t}} \beta_s^{(t)}+1)^{-2} \prod_{k=1}^{t} \tilde{\alpha}_{i_k}^{-a_{i_k,i_{t+1}}}+1 },
\end{align}
and 
\begin{align}\label{alpha}
    \tilde{\alpha}_{i_{t+1}}=c_{t+1}^2 \prod_{k=1}^{t} \tilde{\alpha}_{i_k}^{-a_{i_k,i_{t+1}}} + (c_{t+1}\sum_{\substack{i_s=i_{t+1}\\1\leq s\leq t}} \beta_s^{(t)}+1)^2.
\end{align}

Inductively, we obtain $\beta_1^{(m)},\cdots,\beta_m^{(m)}\geq 0$ and $\tilde{\alpha}_{i_1},\cdots,\tilde{\alpha}_{i_m}>0$, such that 
\begin{align*}
    &E_{[S_{i_m}]}(c_m) \cdots E_{[S_{i_1}]}(c_1)E_{[TS_{i_1}]}(-c_1)\cdots E_{[TS_{i_m}]}(-c_m)\\
    =&E_{[TS_{i_1}]}(-\beta_1^{(m)})\cdots E_{[TS_{i_m}]}(-\beta_m^{(m)}) h_{i_1}(\tilde{\alpha}_{i_1})\cdots h_{i_m}(\tilde{\alpha}_{i_m}) E_{[S_{i_m}]}(\beta_m^{(m)}) \cdots E_{[S_{i_1}]}(\beta_1^{(m)}).
\end{align*}

For $1\leq k\leq m$, let 
\begin{align*}
    \beta_k=\beta_k^{(m)}\prod_{j=1}^{n} b_j^{-a_{j,i_k}}.
\end{align*}

For $1\leq j\leq n$, let 
\begin{align*}
    \alpha_j=\prod_{\substack{1\leq k\leq m\\i_k=j}} \tilde{\alpha}_{i_k} b_j^2.
\end{align*}

We summarize the above calculations as the following proposition.

\begin{proposition}\label{ab}
    For any $b_1,\cdots,b_n\in \mathbb{R}_{>0}$ and $c_1,\cdots,c_m\in \mathbb{R}_{\geq 0}$, there exist $\alpha_1,\cdots,\alpha_n\in \mathbb{R}_{>0}$ and $\beta_1,\cdots,\beta_m \in \mathbb{R}_{\geq 0}$, such that 
\begin{align*}
    &h_1(b_1)\cdots h_n(b_n) E_{[S_{i_m}]}(c_m) \cdots E_{[S_{i_1}]}(c_1)E_{[TS_{i_1}]}(-c_1)\cdots E_{[TS_{i_m}]}(-c_m) h_1(b_1)\cdots h_n(b_n) \\
    &= E_{[TS_{i_1}]}(-\beta_1)\cdots E_{[TS_{i_m}]}(-\beta_m)h_1(\alpha_1)\cdots h_n(\alpha_n) E_{[S_{i_m}]}(\beta_m)\cdots E_{[S_{i_1}]}(\beta_1),
\end{align*}
and these numbers can be inductively calculated via rational functions of $b_1,\cdots,b_n$, $c_1,\cdots,c_m$.
\end{proposition}

The proposition gives a map $\phi:\mathbb{R}_{\geq 0}^m \times \mathbb{R}_{>0}^n \rightarrow \mathbb{R}_{\geq 0}^m \times \mathbb{R}_{>0}^n$,
\begin{align*}
    (c_1,\cdots,c_m,b_1,\cdots,b_n)\mapsto (\beta_1,\cdots,\beta_m,\alpha_1,\cdots,\alpha_n),
\end{align*}
with the notations as above.

\begin{proposition}\label{phi}
    The map $\phi$ is injective.
\end{proposition}

\begin{proof}
    For any $(\beta_1,\cdots,\beta_m,\alpha_1,\cdots,\alpha_n)\in \operatorname{Im}\phi$, there exists $(c_1,\cdots,c_m,b_1,\cdots,b_n) \in \mathbb{R}_{\geq 0}^m \times \mathbb{R}_{>0}^n$ such that $\phi(c_1,\cdots,c_m,b_1,\cdots,b_n)=(\beta_1,\cdots,\beta_m,\alpha_1,\cdots,\alpha_n)$.
    Let $\beta_k^{(t)},\tilde{\alpha}_{i_k}$ be defined by $(c_1,\cdots,c_m,b_1,\cdots,b_n)$ as the process above.
    Since 
    \begin{align*}
        \beta_k \prod_{j=1}^{n} \alpha_j^{\frac{a_{j,i_k}}{2}} = \prod_{j=1}^{n} (\frac{\sqrt{\alpha_j}}{b_j})^{a_{j,i_k}} \beta_k^{(m)},
    \end{align*}
    we have 
    \begin{align*}
        &E_{[S_{i_m}]}(c_m)\cdots E_{[S_{i_1}]}(c_1) E_{[TS_{i_1}]}(-c_1)\cdots E_{[TS_{i_m}]}(-c_m)\\
        =&E_{[TS_{i_1}]}(-\beta_1^{(m)})\cdots E_{[TS_{i_m}]}(-\beta_m^{(m)}) h_1(\prod_{\substack{1\leq k\leq m\\i_k=1}} \tilde{\alpha}_{i_k}) \cdots h_n(\prod_{\substack{1\leq k\leq m\\i_k=n}} \tilde{\alpha}_{i_k})\\
        & \times  E_{[S_{i_m}]}(\beta_{m}^{(m)})\cdots E_{[S_{i_1}]}(\beta_1^{(m)}) \\
        =& h_1(\frac{\sqrt{\alpha_1}}{b_1})\cdots h_n(\frac{\sqrt{\alpha_n}}{b_n}) E_{[TS_{i_1}]}(- \beta_1 \prod_{j=1}^{n} \alpha_j^{\frac{a_{j,i_1}}{2}}) \cdots E_{[TS_{i_m}]}(-  \beta_m \prod_{j=1}^{n} \alpha_j^{\frac{a_{j,i_m}}{2}}) \\
        & \times E_{[S_{i_m}]}( \beta_m \prod_{j=1}^{n} \alpha_j^{\frac{a_{j,i_m}}{2}}) \cdots E_{[S_{i_1}]}( \beta_1 \prod_{j=1}^{n} \alpha_j^{\frac{a_{j,i_1}}{2}}) h_1(\frac{\sqrt{\alpha_1}}{b_1})\cdots h_n(\frac{\sqrt{\alpha_n}}{b_n}).
    \end{align*}
    On the other hand, by a similar approach, we can calculate 
    \begin{align*}
         E_{[TS_{i_1}]}(- \beta_1 \prod_{j=1}^{n} \alpha_j^{\frac{a_{j,i_1}}{2}}) \cdots E_{[TS_{i_m}]}(-  \beta_m \prod_{j=1}^{n} \alpha_j^{\frac{a_{j,i_m}}{2}})  E_{[S_{i_m}]}( \beta_m \prod_{j=1}^{n} \alpha_j^{\frac{a_{j,i_m}}{2}}) \cdots E_{[S_{i_1}]}( \beta_1 \prod_{j=1}^{n} \alpha_j^{\frac{a_{j,i_1}}{2}}),
    \end{align*}
    and obtain $d_1,\cdots,d_m \in \mathbb{R}_{\geq 0}$ and $h\in \mathbf{H}_{>0}$ such that the element equals to 
    \begin{align*}
        E_{[S_{i_m}]}(d_m)\cdots E_{[S_{i_1}]}(d_1) h E_{[TS_{i_1}]}(-d_1)\cdots E_{[TS_{i_m}]}(-d_m).
    \end{align*}
    Thus we have 
    \begin{align*}
        &E_{[S_{i_m}]}(c_m)\cdots E_{[S_{i_1}]}(c_1) E_{[TS_{i_1}]}(-c_1)\cdots E_{[TS_{i_m}]}(-c_m)\\
        =& E_{[S_{i_m}]}(\prod_{j=1}^{n}(\frac{\sqrt{\alpha_j}}{b_j})^{a_{j,i_m}}d_m)\cdots E_{[S_{i_1}]}(\prod_{j=1}^{n}(\frac{\sqrt{\alpha_j}}{b_j})^{a_{j,i_1}}d_1) h_1(\frac{\alpha_1}{b_1^2})\cdots h_n(\frac{\alpha_n}{b_n^2}) h \\
        & \times E_{[TS_{i_1}]}(-\prod_{j=1}^{n}(\frac{\sqrt{\alpha_j}}{b_j})^{a_{j,i_1}}d_1)\cdots E_{[TS_{i_m}]}(-\prod_{j=1}^{n}(\frac{\sqrt{\alpha_j}}{b_j})^{a_{j,i_m}}d_m).
    \end{align*}
    By Lemma \ref{utu}, $h_1(\frac{\alpha_1}{b_1^2})\cdots h_n(\frac{\alpha_n}{b_n^2}) h=1$.
    Since $h$ is uniquely determined by $\beta_1,\cdots,\beta_m$, $\alpha_1,\cdots,\alpha_n$, we have $b_1,\cdots,b_n$ are uniquely determined by $\beta_1,\cdots,\beta_m,\alpha_1,\cdots,\alpha_n$.
    Moreover, the approach above gives an inverse of $\phi$, which shows that $\phi$ is a bijection from $\mathbb{R}_{\geq 0}^m \times \mathbb{R}_{>0}^n$ to $\operatorname{Im}\phi$.
    Thus $\phi$ is injective.
\end{proof}

\begin{theorem}
    The Iwasawa decomposition of $\mathbf{G}_{\geq 0}$ has its $\mathbf{A}$ and $\mathbf{U}^+$ component contained in $\mathbf{G}_{\geq 0}$.
\end{theorem}

\begin{proof}
    For any element $z\in \mathbf{G}_{\geq 0}$, by Lemma \ref{utu} and Lemma \ref{u>=0}, it can be written of the form $z=E_{[TS_{i_1}]}(-c_1)\cdots E_{[TS_{i_m}]}(-c_m)h_1(b_1)\cdots h_n(b_n)u_{\geq 0}^+$, where $(i_1,\cdots,i_m)$ is a reduced expression of $\overline{n}_0$, $c_1,\cdots,c_m\geq 0$, $b_1,\cdots,b_n>0$, and $u_{\geq 0}^+\in \mathbf{U}^+_{\geq 0}$.

    By Iwasawa decomposition of $\mathbf{G}$, there exists $u\in \mathbf{U}^+$ and $a\in \mathbf{A}$, such that $\Theta(zu^{-1}a^{-1})=zu^{-1}a^{-1}$.
    That is, 
    \begin{align*}
        &E_{[S_{i_1}]}(-c_1)\cdots E_{[S_{i_m}]}(-c_m)h_1(b_1^{-1})\cdots h_n(b_n^{-1}) \Theta(u_{\geq 0}^+u^{-1})a\\
        =& E_{[TS_{i_1}]}(-c_1)\cdots E_{[TS_{i_m}]}(-c_m)h_1(b_1)\cdots h_n(b_n)u_{\geq 0}^+ u^{-1}a^{-1}.
    \end{align*}
    Then we have 
    \begin{align*}
        &\Theta(u_{\geq 0}^+u^{-1})a^2 (u_{\geq 0}^+ u^{-1})^{-1}\\
        =& h_1(b_1)\cdots h_n(b_n) E_{[S_{i_m}]}(c_m) \cdots E_{[S_{i_1}]}(c_1)E_{[TS_{i_1}]}(-c_1)\cdots E_{[TS_{i_m}]}(-c_m) h_1(b_1)\cdots h_n(b_n).
    \end{align*}

    Using Proposition \ref{ab}, there exist $\alpha_1,\cdots,\alpha_n\in \mathbb{R}_{>0}$ and $\beta_1,\cdots,\beta_m \in \mathbb{R}_{\geq 0}$, calculated by $c_1,\cdots,c_m$, $b_1,\cdots,b_n$, and satisfying
\begin{align*}
    &h_1(b_1)\cdots h_n(b_n) E_{[S_{i_m}]}(c_m) \cdots E_{[S_{i_1}]}(c_1)E_{[TS_{i_1}]}(-c_1)\cdots E_{[TS_{i_m}]}(-c_m) h_1(b_1)\cdots h_n(b_n) \\
    &= E_{[TS_{i_1}]}(-\beta_1)\cdots E_{[TS_{i_m}]}(-\beta_m)h_1(\alpha_1)\cdots h_n(\alpha_n) E_{[S_{i_m}]}(\beta_m)\cdots E_{[S_{i_1}]}(\beta_1).
\end{align*}

Then it is easy to see that 
\begin{align*}
    & u=E_{[S_{i_m}]}(\beta_m)\cdots E_{[S_{i_1}]}(\beta_1)u_{\geq 0}^+ \in \mathbf{U}_{\geq 0}^+ ,\\
    & a=h_1(\alpha_1^{\frac{1}{2}})\cdots h_n(\alpha_n^{\frac{1}{2}}) \in \mathbf{H}_{>0},
\end{align*}
satisfies $\Theta(zu^{-1}a^{-1})=zu^{-1}a^{-1}$.
There exists $k\in \mathbf{K}$ such that $k=zu^{-1}a^{-1}$, or equivalently, $z=kau$.
By the uniqueness of Iwasawa decomposition, $k,a,u$ are exactly the $\mathbf{K},\mathbf{A},\mathbf{U}^+$ components of $z$.
It is clear that the $\mathbf{A},\mathbf{U}^+$ components of $z$ are still in $\mathbf{G}_{\geq 0}$, and the theorem is proved.
\end{proof}

Note that the proofs of the previous proposition and theorem  give an explicit algorithm for computing the $\mathbf{A}$ and $\mathbf{U}^+$ components of elements in $\mathbf{G}_{\geq 0}$.
Although we have not given an explicit formula for the $\mathbf{K}$ component of $\mathbf{G}_{\geq 0}$, we have the following corollary concerning $\mathbf{K}$, which shows that the intersection of totally positive monoid and the maximal compact subgroup is trivial.

\begin{theorem}\label{GK=1}
    The intersection $\mathbf{G}_{\geq 0}\bigcap \mathbf{K}=\{1\}$.
\end{theorem}

\begin{proof}
    For any element $z\in \mathbf{G}_{\geq 0}\bigcap \mathbf{K}$ and write its Iwasawa decomposition $z=kau$, $k\in \mathbf{K},a\in \mathbf{A}, u\in \mathbf{U}^+$, then $a=u=1$.
    Following the same process as the proof of the previous theorem and keep the notations, we have 
    \begin{align*}
    & 1=u=E_{[S_{i_m}]}(\beta_m)\cdots E_{[S_{i_1}]}(\beta_1)u_{\geq 0}^+  ,\\
    & 1=a=h_1(\alpha_1^{\frac{1}{2}})\cdots h_n(\alpha_n^{\frac{1}{2}}) .
\end{align*}

Then $u_{\geq 0}^+=E_{[S_{i_1}]}(-\beta_1)\cdots E_{[S_{i_m}]}(-\beta_m)$.
Since $u_{\geq 0}^+ \in \mathbf{U}_{\geq 0}^+$ and $\beta_1,\cdots,\beta_m\geq 0$, we have $\beta_1=\cdots=\beta_m=0$ and $u_{\geq 0}^+=1$.

Since $h_1(\alpha_1^{\frac{1}{2}})\cdots h_n(\alpha_n^{\frac{1}{2}})=1$ and $\alpha_1,\cdots,\alpha_n>0$, we take logarithms of the equations $\prod_{i=1}^n \alpha_i^{\frac{a_{ij}}{2}}=1$ for $j=1,\cdots,n$ and use the fact that the Cartan matrix is invertible, and then obtain that $\alpha_1=\cdots=\alpha_n=1$.

For $1\leq k\leq m$, since $0=\beta_k=\beta_k^{(m)}\prod_{j=1}^{n}b_j^{-a_{j,i_k}}$, we have $\beta_k^{(m)}=0$.

Firstly, consider $\beta_m^{(m)}=0$.
Then the numerator of formula (\ref{beta}) is zero (case $t+1=m$).
Since $\tilde{\alpha}_{i_k}>0$ and $c_{m}\sum_{\substack{i_s=i_{m}\\1\leq s\leq m-1}} \beta_s^{(m-1)}+1>0$, we have $c_m=0$.
In fact, for any $k$, if $\beta_k^{(k)}=0$, then $c_k=0$ by the same argument.

Now consider $k<m$ and $\beta_k^{(m)}=0$.
By the inductive formulas (\ref{inductive1}) and (\ref{inductive2}), we have $\beta_k^{(m-1)}=\cdots=\beta_k^{(k)}=0$, and $c_k=0$ follows.

So we have $c_1=\cdots=c_m=0$.
By formula (\ref{alpha}), we have $\tilde{\alpha}_{i_k}=1$ for all $1\leq k\leq m$.
Then $b_j^2=1$, and thus $b_j=1$ for all $1\leq j\leq n$.

As a result, we have $z=E_{[TS_{i_1}]}(-c_1)\cdots E_{[TS_{i_m}]}(-c_m)h_1(b_1)\cdots h_n(b_n)u_{\geq 0}^+=1$, and the theorem is proved.
\end{proof}

Let $\psi : \mathbf{G}_{\geq 0} \rightarrow \mathbf{H}_{>0} \times \mathbf{U}^+_{\geq 0}$ be the map induced by Iwasawa decomposition.
Note that the map $\mathbb{R}_{\geq 0}^m \rightarrow \mathbf{U}_{\geq 0}^+$ given in Lemma \ref{u>=0} is not injective in general.
Thus $\psi$ is not necessarily injective.
We consider the special case that $u_{\geq 0}^+=1$.

\begin{proposition}
    The restriction $\psi|_{\mathbf{U}_{\geq 0}^- \mathbf{H}_{>0}}: \mathbf{U}_{\geq 0}^- \mathbf{H}_{>0} \rightarrow \mathbf{H}_{>0}\times \mathbf{U}_{\geq 0}^+$ is injective.
\end{proposition}

\begin{proof}
    Note that by the same argument in the proof of Theorem \ref{GK=1}, we have if $\phi(c_1,\cdots,c_m,b_1,\cdots,b_n)=(\beta_1,\cdots,\beta_m,\alpha_1,\cdots,\alpha_n)$, then for any $1\leq k\leq m$, $c_k=0$ if and only if $\beta_k=0$.
    Thus we can consider the restriction of $\psi$ to $\mathbf{U}^-(\overline{n})\mathbf{H}_{>0}$ for any $n\in \mathbf{N}$.

    Precisely, for any $z\in \mathbf{U}_{\geq 0}^- \mathbf{H}_{>0}= \bigsqcup_{\overline{n}\in \mathbf{N}/\mathbf{H}} \mathbf{U}^-(\overline{n})\mathbf{H}_{>0} $, there exists a unique $\overline{n}$ such that $z\in \mathbf{U}^-(\overline{n})\mathbf{H}_{>0}$.
    Let $(i_1,\cdots,i_t)$ be a reduced expression of $\overline{n}$, and complete it into a reduced expression $(i_1,\cdots,i_t,i_{t+1},\cdots,i_m)$ of $\overline{n}_0$.
    By Proposition \ref{un}, there uniquely exists $c_1,\cdots,c_t,b_1,\cdots,b_n \in \mathbb{R}_{>0}$, such that 
    \begin{align*}
        z=&E_{[TS_{i_1}]}(-c_1)\cdots E_{[TS_{i_t}]}(-c_t) h_1(b_1)\cdots h_n(b_n)\\
        =&  E_{[TS_{i_1}]}(-c_1)\cdots E_{[TS_{i_t}]}(-c_t) E_{[TS_{i_{t+1}}]}(0)\cdots E_{[TS_{i_m}]}(0) h_1(b_1)\cdots h_n(b_n).
    \end{align*}
    Let $ (\beta_1,\cdots,\beta_t,0,\cdots,0,\alpha_1,\cdots,\alpha_n)=\phi(c_1,\cdots,c_t,0,\cdots,0,b_1,\cdots,b_n),$ and we have $\beta_1,\cdots,\beta_t \in \mathbb{R}_{>0}$.
    Then the image $(a,u)=\psi(z)$ is $a=h_1(\alpha_1^{\frac{1}{2}})\cdots h_n(\alpha_n^{\frac{1}{2}})$ and 
    \begin{align*}
        u=E_{[S_{i_m}]}(0)\cdots E_{[S_{i_{t+1}}]}(0) E_{[S_{i_t}]}(\beta_t)\cdots E_{[S_{i_1}]}(\beta_1) \in \mathbf{U}^+(\overline{n^{-1}}).
    \end{align*}
  By Proposition \ref{un} and \ref{phi}, the restriction $\psi|_{\mathbf{U}^-(\overline{n}) \mathbf{H}_{>0}} : \mathbf{U}^-(\overline{n}) \mathbf{H}_{>0} \rightarrow \mathbf{H}_{>0}\times \mathbf{U}^+(\overline{n^{-1}})$ is injective. 
  Thus $\psi|_{\mathbf{U}_{\geq 0}^- \mathbf{H}_{>0}}: \mathbf{U}_{\geq 0}^- \mathbf{H}_{>0} =   \bigsqcup_{\overline{n}\in \mathbf{N}/\mathbf{H}} \mathbf{U}^-(\overline{n})\mathbf{H}_{>0} \rightarrow \mathbf{H}_{>0}\times \mathbf{U}_{\geq 0}^+$ is injective.
\end{proof}

As an application, the fibre of $\psi$ at any $(a,u)\in \operatorname{Im}\psi$ is 
\begin{align*}
    \{ v\in \mathbf{U}_{\geq 0}^+ \mid (a,v)\in \operatorname{Im}(\psi|_{\mathbf{U}_{\geq 0}^- \mathbf{H}_{>0}}), v^{-1}u \in \mathbf{U}_{\geq 0}^+  \}.
\end{align*}

\begin{example}
    We consider the case of type $A_1$.
    Let $S$ be the unique (up to isomorphism) simple object of the fixed hereditary subcategory $\mathcal{B}$ of $\mathcal{R}$.
    Any element $z\in \mathbf{G}_{\geq 0}$ is of the form $E_{[TS]}(-c) h_{[S]}(b) E_{[S]}(d)$ for some $c,d\geq 0$ and $b>0$.
    We can calculate that 
    \begin{align*}
        \beta = \frac{cb^{-2}}{c^2+1}, \quad \alpha = (c^2+1)b^2,
    \end{align*}
    and then $\psi(z)=(a,u)$ with 
    \begin{align*}
        &u=E_{[S]}(\beta)E_{[S]}(d)=E_{[S]}(\frac{cb^{-2}}{c^2+1}+d), \\
        & a= h_{[S]}(b\sqrt{c^2+1}).
    \end{align*}

    Conversely, for any $(a,u)\in \operatorname{Im}\psi$, assume that $a=h_{[S]}(p),u=E_{[S]}(q)$.
    Then for any $0\leq d \leq q$, we have 
    \begin{align*}
        c=p^2(q-d), \quad b= \frac{p}{\sqrt{p^4(q-d)^2+1}},
    \end{align*}
    and $z=E_{[TS]}(-p^2(q-d)) h_{[S]}(\frac{p}{\sqrt{p^4(q-d)^2+1}}) E_{[S]}(d)$.
    Thus the fibre of $\psi$ at $(a,u)=(h_{[S]}(p),E_{[S]}(q))$ is $[0,q]$.
\end{example}

\subsection*{Acknowledgments}
This paper is a continuation of a series of adjoint works with Professor Jie Xiao. 
The author would like to thank Professor Jie Xiao for his guidance and valuable discussions.
The author would also like to thank Yixin Lan for the discussions.
The author was partially supported by National Natural Science Foundation of China [Grant No. 12471030].


\begin{thebibliography}{1}

\bibitem{Knapp}
A.~W. Knapp.
\newblock {\em Lie groups beyond an introduction}, volume 140 of {\em Progress
  in Mathematics}.
\newblock Birkh\"auser Boston, Inc., Boston, MA, second edition, 2002.

\bibitem{notes1}
B.~Li and J.~Xiao.
\newblock Notes on {C}hevalley groups and root category {I}.
\newblock {\em arXiv:2505.17805}, 2025.
\newblock In press.

\bibitem{notes3}
B.~Li and J.~Xiao.
\newblock Notes on {C}hevalley groups and root category {III}: the region of
  total positivity.
\newblock {\em Sci China Math}, 2026, 69, https://doi.org/10.1007/s11425-026-2666-9

\bibitem{notes2}
B.~Li and J.~Xiao.
\newblock Notes on {C}hevalley groups and root category {II}: Compact {L}ie
  groups and representations.
\newblock {\em Journal of Algebra}, 711:113--155, 2027.

\bibitem{TR1}
G.~Lusztig.
\newblock Total positivity in reductive groups.
\newblock In {\em Lie theory and geometry}, volume 123 of {\em Progr. Math.},
  pages 531--568. Birkh\"auser Boston, Boston, MA, 1994.

\bibitem{TC}
G.~Lusztig.
\newblock Total positivity and canonical bases.
\newblock In {\em Algebraic groups and {L}ie groups}, volume~9 of {\em Austral.
  Math. Soc. Lect. Ser.}, pages 281--295. Cambridge Univ. Press, Cambridge,
  1997.

\bibitem{TR2}
G.~Lusztig.
\newblock Total positivity in reductive groups, {II}.
\newblock {\em Bull. Inst. Math. Acad. Sin. (N.S.)}, 14(4):403--459, 2019.

\bibitem{1997ROOT}
L.~Peng and J.~Xiao.
\newblock Root categories and simple lie algebras.
\newblock {\em Journal of Algebra}, 198(1):19--56, 1997.

\bibitem{RINGEL1990137}
C.~M. Ringel.
\newblock Hall polynomials for the representation-finite hereditary algebras.
\newblock {\em Adv. Math.}, 84(2):137--178, 1990.

\end{thebibliography}
\end{document}